\documentclass[11pt]{amsart}
\usepackage{amscd}
\usepackage[latin1]{inputenc}
\usepackage{amsmath}
\usepackage{amssymb}
\usepackage{amsthm}
\usepackage{latexsym}
\usepackage{pstcol,pst-plot,pst-3d}

\usepackage{mathptmx}
\usepackage{multicol}
\usepackage{hyperref}

\psset{unit=0.7cm,linewidth=0.8pt,arrowsize=2.5pt 4}

\newpsstyle{fatline}{linewidth=1.5pt}
\newpsstyle{fyp}{fillstyle=solid,fillcolor=verylight}
\definecolor{verylight}{gray}{0.97}
\definecolor{light}{gray}{0.9}
\definecolor{medium}{gray}{0.85}

\def\NZQ{\Bbb}               
\def\NN{{\NZQ N}}

\def\ZZ{{\NZQ Z}}

\def\frk{\frak}               

\def\Phi{{\frk n}}
\def\Phi{{\frk N}}
\def\opn#1#2{\def#1{\operatorname{#2}}} 
\opn\chara{char} \opn\length{\ell} \opn\pd{pd} \opn\rk{rk}
\opn\projdim{proj\,dim} \opn\injdim{inj\,dim} \opn\rank{rank}
\opn\depth{depth} \opn\grade{grade} \opn\height{height}
\opn\embdim{emb\,dim} \opn\codim{codim} \opn\dim{dim}
\opn\sdepth{sdepth} \opn\sqdepth{sqdepth}

\opn\Tr{Tr} \opn\bigrank{big\,rank}
\opn\superheight{superheight}\opn\lcm{lcm}
\opn\trdeg{tr\,deg}%
\opn\reg{reg} \opn\lreg{lreg} \opn\ini{in} \opn\Syz{Syz}
\opn\sq{sq}
\opn\div{div} \opn\Div{Div} \opn\cl{cl} \opn\Cl{Cl}
\opn\Spec{Spec} \opn\Supp{Supp} \opn\supp{supp} \opn\Sing{Sing}
\opn\Ass{Ass}  \opn\Min{Min} \opn\max{max}
\opn\Ann{Ann} \opn\Rad{Rad} \opn\Soc{Soc}
\opn\Ker{Ker} \opn\Coker{Coker} \opn\Am{Am} \opn\Hom{Hom}
\opn\Tor{Tor} \opn\Ext{Ext} \opn\End{End} \opn\Aut{Aut}
\opn\id{id}  \opn\e{e} \opn\hreg{hreg} \opn\dim{dim}

\opn\nat{nat} \opn\deg{deg} \opn\adeg{adeg} \opn\Hilb{Hilb}
\opn\pff{pf}
\opn\Pf{Pf} \opn\GL{GL} \opn\SL{SL} \opn\mod{mod} \opn\ord{ord}
\opn\aff{aff} \opn\con{conv} \opn\relint{relint} \opn\st{st}
\opn\lk{lk} \opn\cn{cn} \opn\core{core} \opn\vol{vol}
\opn\link{link}  \opn\infpt{infpt} \opn\mult{mult} \opn\Sqp{Sqp}
\opn\gr{gr} \opn\Std{Std}

\def\pot#1#2{#1[\kern-0.28ex[#2]\kern-0.28ex]}

\opn\dirlim{\underrightarrow{\lim}}
\opn\inivlim{\underleftarrow{\lim}}
\let\union=\cup
\let\sect=\cap
\let\dirsum=\oplus

\let\iso=\cong
\let\Union=\bigcup

\let\Dirsum=\bigoplus

\let\to=\rightarrow

\def\Implies{\ifmmode\Longrightarrow \else
     \unskip${}\Longrightarrow{}$\ignorespaces\fi}
\def\implies{\ifmmode\Rightarrow \else
     \unskip${}\Rightarrow{}$\ignorespaces\fi}
\def\iff{\ifmmode\Longleftrightarrow \else
     \unskip${}\Longleftrightarrow{}$\ignorespaces\fi}

\let\:=\colon
\newtheorem{Theorem}{Theorem}[section]
\newtheorem{Lemma}[Theorem]{Lemma}
\newtheorem{Corollary}[Theorem]{Corollary}
\newtheorem{Proposition}[Theorem]{Proposition}
\newtheorem{Remark}[Theorem]{Remark}

\newtheorem{Example}[Theorem]{Example}

\newtheorem{Definition}[Theorem]{Definition}

\newtheorem{Conjecture}[Theorem]{Conjecture}

\let\epsilon\varepsilon
\let\phi=\varphi
\let\kappa=\varkappa
\opn\dis{dis}
\def\pnt{{\raise0.5mm\hbox{\large\bf.}}}

\begin{document}

\title{Stanley decompositions of squarefree modules and Alexander duality }
\author{Ali Soleyman Jahan}
\address{Ali Soleyman Jahan, Fachbereich Mathematik und
Informatik, Universit\"at Duisburg-Essen, 45117 Essen, Germany}
\email{ali.soleyman-jahan@stud.uni-duisburg- essen.de}

\date{}
\begin{abstract} In this paper we study how  prime filtrations and
 squarefree Stanley decompositions of  squarefree modules over
the polynomial ring and the exterior algebra behave with respect
to Alexander duality.
\end{abstract}

 \maketitle

\section*{Introduction} Let $K$ be a field and $S=K[x_1,\ldots,x_n]$ the
polynomial ring in $n$ variables. The ring $S$ is naturally
$\NN^n$-graded. Yanagawa \cite{Y} introduced
squarefree $S$-modules which generalizes the concept of
Stanley--Reisner rings. A finitely generated $\NN^n$-graded
$S$-module $M=\Dirsum_{\mathbf a\in\NN^n}M_\mathbf{a}$ is {\em
squarefree} if the multiplication map $M_\mathbf{a}\to M_{\mathbf
a+\epsilon_i}$, $m\mapsto mx_i$, is bijective for all $\mathbf
a\in\NN^n$ and all $i\in\supp(\mathbf a)$.
R\"{o}mer defined in \cite{R} the Alexander dual $M^\vee$ for a
squarefree $S$-module $M$. The definition refers to exterior
algebras. Let $E$ be the exterior algebra over an $n$-dimensional
$K$-vector space $V$. A finitely generated $\NN^n$-graded
$E$-module $N=\Dirsum_{\mathbf a\in\NN^n}N_\mathbf{a}$ is called
{\em squarefree} if it has only squarefree components. By
\cite[Corollary 1.6]{R}
 the category of squarefree
$S$-modules is equivalent to the category of squarefree
$E$-modules. For an $\NN^n$-graded $E$-module $N$ the $E$-dual of
$N$ is the graded dual $N^\vee=\Hom_E(N,E)$. Let $M$ be a
squarefree $S$-module and $N$ its corresponding  squarefree
$E$-module. Then $M^\vee$ is defined to be the squarefree
$S$-module corresponding to $N^\vee$. In the first section of this
paper we recall some basic notion and definitions about squarefree
$S$-modules and $E$-modules. In Section 2 we study prime
filtrations of squarefree $S$-modules and $E$-modules. As a main
result of this section we prove that for a squarefree $S$-module
$M$ there exists a chain $ 0\subset M_1\subset\cdots\subset M_r=M$
of squarefree submodules of $M$ with $M_i/M_{i-1}\iso
S/P_{F_i}(-G_i)$ if and only if there exists a chain $0\subset
L_1\subset \cdots\subset L_r=M^\vee$ of  squarefree submodules of
$M^\vee$ with $L_i/ L_{i-1}\iso S/P_{G_i}(-F_i)$, see Theorem
\ref{psdual}. For proving this, in Proposition \ref{pedual} we
show that the corresponding result is true for squarefree
$E$-modules. In Corollary \ref{J/I} we show explicitly how the
prime filtration of $M^\vee$ is obtained form that of $M$, in the
special case that $M=J/I$, where $I\subset J$ are squarefree
monomial ideals.

In last section we study Stanley decompositions of finitely
generated $\ZZ^n$-graded $S$-modules. Let $m\in M$ be a
homogeneous element and $Z\subset\{x_1,\ldots,x_n\}=X$. We denote
by $mK[Z]$ the $K$-subspace of $M$ generated by all homogeneous
elements of the form $mu$, where $u$ is a monomial in $K[Z]$. The
$K$-subspace $mK[Z]$ is called a {\em Stanley space of dimension
$|Z|$} if $mu\neq 0$ for all nonzero monomial $u\in K[Z]$. Here
$|Z|$ denote the cardinality of $Z$. A homogeneous element $m\in
M$ is called squarefree if $\deg(m)=(a_1,\ldots,a_n)\in
\{0,1\}^n$.
The Stanley space $mK[Z]$ is called {\em squarefree}
if $m$ is a squarefree homogeneous element and
$\supp(\deg(m))=\{j\:a_j\neq 0\} \subset\{i\:x_i\in Z\}$.

A decomposition $\mathcal D$ of $M$ as a finite direct sum of
Stanley spaces is called a {\em Stanley decomposition} of $M$. The
Stanley decomposition $\mathcal D$ of $M$  is called {\em
squarefree Stanley decomposition} if all Stanley spaces in
$\mathcal D$ are squarefree Stanley spaces. In Proposition
\ref{squarefree} we show that the $R$-module $M$ has a squarefree
Stanley decomposition if and only if $M$ is squarefree $R$-module.
The minimal dimension of a Stanley space in the decomposition
$\mathcal D$ is called the {\em Stanley depth} of $\mathcal D$,
denoted by $\sdepth ({\mathcal D})$. We set
$$\sdepth(M)=\max\{\sdepth({\mathcal D})\: {\mathcal D}\; \text{is a
Stanley decomposition of $M$}\}, $$ and call this number the {\em
Stanley depth} of $M$. For a squarefree module $M$ we denote by
$$\sqdepth(M)=\max\{\sdepth({\mathcal D})\: {\mathcal D} \text{ is
a squarefree Stanley decomposition of $M$}\}$$ the {\em squarefree
Stanley depth} of $M$. If $M$ is squarefree, then
$\sqdepth(M)=\sdepth(M)$, see Theorem ~\ref{sqdepth}.

As a main result of this section we show that a squarefree
$S$-module $M$ has a squarefree Stanley decomposition
$M=\Dirsum_{i=1}^tm_iK[Z_i]$  if and only if there exist a
squarefree Stanley decomposition $M^\vee=\Dirsum_{i=1}^tv_iK[W_i]$
of $M^\vee$ with $\supp(v_i)=[n]\setminus\{j\:x_j\in Z_i\}$ and
$W_i=\{x_j\:j\in[n]\setminus\supp(m_i)\}$, see Theorem \ref{dual}.
To prove this we show in Proposition \ref{Edual} that the
correspomding result is true for squarefree $E$-modules. As
corollaries of Theorem \ref{dual} we show that Stanley's
conjecture on Stanley decompositions holds for a squarefree
$S$-module $M$ if and only if $M^\vee$ has a Stanley decomposition
$M^\vee=\Dirsum_{i=1}^tv_iK[W_i]$ with $|v_i|\leq\reg(M^\vee)$ for
all $i$, see Corollary \ref{1}, and Stanley's conjecture on
partitionable simplicial complexes holds for a Cohen--Macaulay
simplicial complex $\Delta$ if and only if $I_{\Delta^\vee}$ has a
Stanley decomposition $I_{\Delta^\vee}=\Dirsum_{i=1}^tu_iK[Z_i]$
such that $\{u_i,\ldots,u_t\}=G(I_{\Delta^\vee})$.

Due to these facts we conjecture (Conjecture \ref{reg}) that any
$\ZZ^n$-graded $S$-module $M$ has a Stanley decomposition
$M=\Dirsum_{i=1}^tm_iK[Z_i]$ with $|m_i|\leq\reg(M)$. In some
cases we can show that this conjecture holds.

\subsection*{Acknowledgements} I am grateful to Professor
J\"{u}rgen Herzog for his suggestion to  consider  Stanley
decompositions of a squarefree $S$-module and their Alexander
dual, and for useful discussions and comments.

\section{squarefree modules and Alexander dual}

We fix some notation and recall some definitions.  For $\mathbf
a=(a_1,\ldots,a_n)\in\ZZ^n$, we say $\mathbf a$ is squarefree  if
$a_i=0$ or $a_i=1$ for $i=1,\ldots,n$. We set $\supp(\mathbf
a)=\{i\: a_i\neq 0\}\subset[n]=\{1,\ldots,n\}$ and $|\mathbf
a|=\sum_{i=1}^na_i$. Occasionally we identify  a squarefree vector
$\mathbf a$ with $\supp(\mathbf a)$. Let
${\epsilon}_i=(0\ldots,1,\ldots,0)\in\NN^n$ be the vector with $1$
at the $i$-th position. Let $M=\Dirsum_{\mathbf
a\in\ZZ^n}M_\mathbf{a}$ be an $\ZZ^n$-graded $K$-vector space. For
simplicity  set $\supp(m)=\supp(\deg m) $ and $|m|=|\deg m|$ for
any homogeneous element $m\in M$ . A homogeneous element $m\in M$
is called {\em squarefree} if $\deg m\in\{0,1\}^n$.

Let $K$ be a field and $S=K[x_1,\ldots,x_n]$ the symmetric algebra
over $K$. Consider the natural $\NN^n$-grading on $S$. For a
monomial $x_1^{a_1}\cdots x_n^{a_n}$ with $\mathbf
a=(a_1,\ldots,a_n)$ we set $x^{\mathbf a}$, and for $F\subset [n]$
we denote $x_F=\prod_{j\in F}x_j$.

Let $V$ be an $n$-dimensional $K$-vector space  with  basis
$e_1,\ldots,e_n$. We denote by $E=K\langle e_1,\ldots,e_n\rangle$
the exterior algebra over $V$. The algebra $E$ is a naturally
$\NN^n$-graded $K$-algebra with $\deg e_i=\epsilon_i$. Let
$F=\{j_1<j_2<\ldots<j_k\}\subset[n]$. Then $e_F=e_{j_1}\wedge
e_{j_2}\wedge\ldots\wedge e_{j_k}$ is called a monomial in $E$. It
is easy to see that the elements $e_F$, with $F\subset [n]$ form a
$K$-basis of $E$. Here we set $e_F=1$, if $F=\emptyset$. For  any
$\mathbf a\in\NN^n$ we set $e_{\mathbf a}=e_{\supp(\mathbf a)}$.

 A finite dimensional  $K$-vector
space $M$ is called an $\ZZ^n$-graded $E$-module, if
\begin{enumerate}
\item[(i)] $M=\Dirsum_{\mathbf a\in \NN^n}M_{\mathbf a}$ is a direct sum of $K$-vector spaces
$M_\mathbf{a}$;
\item[(ii)] $M$ is an $(E-E)$-bimodule;
\item[(iii)] for all vectors $\mathbf a$ and $\mathbf b$ in $\ZZ^n$ and all $f\in E_\mathbf{a}$ and
$m\in M_\mathbf{b}$ one has $fm\in M_{\mathbf a+\mathbf b}$ and
$fm=(-1)^{|\mathbf a||\mathbf b|}mf$.
\end{enumerate}

A simplicial complex $\Delta$ is a collection of subset of
$[n]=\{1,\ldots,n\}$ such that whenever  $F\in \Delta$ and  $G\subset F$, then $G\in\Delta$.
Further we denote by $\Delta^\vee=\{F\: F^c\not\in\Delta\}$ the
Alexander dual of $\Delta$, where $F^c=[n]\setminus F$. Then
$K[\Delta]=S/I_\Delta$ is called the Stanley--Reisner ring, where
$I_\Delta=(x_{i_1}\cdots x_{i_k}\:
\{i_1,\ldots,i_k\}\not\in\Delta)$, and $K\{\Delta\}=E/J_\Delta$
is called the exterior face ring of $\Delta$, where
$J_\Delta=(e_{i_1}\wedge\cdots \wedge e_{i_k}\:
\{i_1,\ldots,i_k\}\not\in\Delta)$ .

The following definition is due to Yanagawa  \cite{Y}.

\begin{Definition} A finitely generated $\NN^n$-graded $S$-module
$M=\Dirsum_{\mathbf a\in\NN^n}M_\mathbf{a}$ is {\em squarefree} if
the multiplication map $M_\mathbf{a}\to M_{\mathbf a+\epsilon_i}$,
$m\mapsto mx_i$, is bijective for all $\mathbf a\in\NN^n$ and all
$i\in\supp(\mathbf a)$.
\end{Definition}

For examples the Stanley-Reisner ring $K[\Delta]$ of a simplicial
complex $\Delta$ is a squarefree $S$-module. If $I\subset J$ are
squarefree monomial ideals, then $I$, $S/I$ and  $J/I$ are
squarefree $S$-modules. The following example shows that the
factor module  $J/I$ may be a squarefree $\NN^n$-graded
$S$-module, even though  $I$ and $J$ are not squarefree monomial
ideals.

\begin{Example}\label{good}{\em Let $I=(x^2,xy)\subset J=(x^2,xy,yz)$ be
monomial ideals in $K[x,y,z]$. Then an element $u\in J\setminus I$
if and only if $u=(yz)v$ for some $v\in K[y,z]$. Hence $J/I$ is a
squarefree $\NN^n$-graded $S$-module.  But if we choose
$I'=(x^2,yz)\subset J=(x^2,xy,yz)\subset K[x,y,z]$, then $xy\in
J\setminus I$ and $x(xy)=x^2y\in I'$. Therefore $J/I'$ is not a
squarefree $\NN^n$-graded $S$-module. }
\end{Example}

Since $\dim_K (J/ I)_{\mathbf{a}}\leq 1$ for all
$\mathbf{a}\in\NN^n$, the $\NN^n$-graded $S$-module $J/I$ is
squarefree if and only if the multiplication map
$$ (J/I)_{\mathbf{a}}\to (J/I)_{\mathbf{a}+\epsilon_i},\; m\to x_im$$ is
injective  for all $i\in\supp(m)$ and all $\mathbf{a}\in\NN^n$.
\begin{Remark} {\em Let $I\subset J\subset S$ be two monomial
ideals. The $\NN^n$-graded $S$-module $J/I$ is squarefree if and
only all minimal monomial generators of $J/I$ are squarefree
monomials and $\supp(u)\not\subset \supp(m)$ for all $m\in
J\setminus I$ and all $u\in G(I)$ where $G(I)$ denote the set of
minimal monomial generators of $I$.  Indeed let $J/I$ be a
squarefree $S$-module and one of the minimal generators of $J/I$
is not sqaurefree, say $m\in J\setminus I$. We may assume that
$x_1^2\mid m$ and $\deg(m)=\mathbf{a}$. Then
$m'=m/x_1\in(J/I)_{\mathbf{a}-\epsilon_i}$ is a zero element and
$1\in\supp(m')$ but $m=x_1m'\in (J/I)_{\mathbf{a}}$ is a nonzero
element, a contradiction. Also if  there exists a monomial $m\in
J\setminus I$ and there exists a  monomial $u\in G(I)$ such that
$\supp(u)\subset\supp(m)$. Then in this case one can find a
minimal monoial  $m'=mx^{{\mathbf a}}$ (with respect to
divisibility) such that $\supp({\mathbf a})\subset\supp(m)$,
$u\mid m'$ and $m'/x_i\not\in I$ for some $i\in
\supp(\mathbf{a})$, again a contradiction.

For the converse  assume that $J/I$ is minimally generated by
squarefree monomials in $J\setminus I$ and  $\supp(u)\not\subset
\supp(m)$ for all $m\in J\setminus I$ and for all $u\in G(I)$. Let
$m\in S$ be a  monomial and $i\in\supp(m)$. Since the minimal
monomial generators of $J/I$ are squarefree, if $m\not\in J$, then
$x_im\not\in J$ or $x_im\in J\sect I$. Hence in this case the
multiplication map $m\to x_im$ is injective. In the case that if
$m\in J\setminus I$, then  $x_im\not\in I$. Otherwise there must
exist a $u\in G(I)$ such that $u\mid x_im$. Therefore
$\supp(u)\subset\supp(x_im)=\supp(m)$ which is a contradiction.}
\end{Remark}

Yanagawa \cite[Lemma 2.3]{Y} proved that if $M$ and $M'$ are
squarefree $S$-modules and $\phi\: M\to M'$ is a
$\NN^n$-homogeneous homomorphism, then $\Ker \phi$ and
$\Coker\phi$ are again squarefree $S$-modules. This implies that
each syzygy module $\Syz_i(M)$ in a multigraded minimal free
$S$-resolution $F_\bullet$ of $M$ is squarefree.

It is easy to see that if $M$ is a squarefree $S$-module, then
$\dim_K M_\mathbf{a}=\dim_K M_{\supp(\mathbf a)}$ for any $\mathbf
a\in\NN^n$, and $M$ is generated by its squarefree parts
$\{M_F\:F\subset[n]\}$.

Next we recall the following definition which is due to T.\;
R\"{o}mer \cite{R}.

\begin{Definition} A finitely generated $\NN^n$-graded $E$-module
$N=\Dirsum_{\mathbf a\in\NN^n}N_\mathbf{a}$ is called squarefree
if it has only squarefree components.
\end{Definition}
For example the exterior face ring $K\{\Delta\}$ of a
simplicial complex $\Delta$ is a squarefree $E$-module.

 We denote by $SQ(S)$ the abelian category of the
squarefree $S$-modules, where the morphisms are the $\NN^n$-graded
homogeneous homomorphisms and denote by $SQ(E)$ the abelian
category of squarefree $E$-modules, where the morphisms are the
$\NN^n$-graded homogeneous homomorphisms. R\"{o}mer
\cite[Corollary 1.6]{R} proved that there are two exact additive
covariant functors
$$ \mathbf{F} \: SQ(S)\mapsto SQ(E),\;\;M\mapsto {\mathbf F}(M)\quad \text{and}\quad
\mathbf G\: SQ(E)\mapsto SQ(S),\;\; N\mapsto {\mathbf G}(N)$$ of
abelian categories such that $(\mathbf F\circ \mathbf G)(N)=N$ and
$(\mathbf G\circ \mathbf F)(M)=M$. Hence the categories $SQ(S)$
and $SQ(E)$ are equivalent. Let $M\in SQ(S)$. By the construction
of $N= \mathbf{F}(M)$ given in \cite{AAH} and \cite{R}, each
minimal homogeneous system of generators $m_1,\ldots,m_t$ of $M$
corresponds to a homogeneous minimal system of generators
$n_1,\ldots,{n_t}$ of $N={\mathbf F}(M)$, and for   all $F\subset
[n]$ we have an isomorphism of $K$-vector spaces $\theta_F\:
M_F\to {\mathbf F}(M)_F$. This isomorphism is described as
follows: an element $m\in M_F$  can be written as $m=\sum
a_im_ix_{F_i}$, where $a_i\in K$  and where $F$ is the disjoint
union of $F_i$ and $\deg(m_i)=G_i$ for each $i$. Then
\begin{eqnarray}
\label{terrible} \theta_F(m) =\sum
(-1)^{\sigma{(G_i,F_i)}}a_in_ie_{F_i},
\end{eqnarray}
where $\sigma{(G_i,F_i)}=|\{(r,s)\: r\in G_i,\; s\in F_i,\;
r>s\}|$. The definition of $\theta_F$ does not depend on the
particular presentation of $m$ as a homogeneous linear combination
of the $m_i$. In particular, we have that $\theta_{G_i}(m_i)=n_i$
for all $i$.

We set  $M_{\sq}=\Dirsum_FM_F$ and define  the isomorphism of
graded $K$-vector spaces $\theta\: M_{\sq}\to N$ by requiring that
$\theta(m)=\theta_F(m)$ for all $m\in M_F$. Now Formula
(\ref{terrible}) can be extended as follows:

\begin{Lemma}
\label{alsoterrible} Let $m$ be a squarefree element of $M$ with
$\supp(m)=F$, and let $m=\sum_i a_iw_ix_{L_i}$ with $a_i\in K$ and
$w_i$ squarefree with $\supp(w_i)=F_i$  such that $F$ is the
disjoint union of $F_i$ and $L_i$ for all $i$. Then
\[
\theta(m) =\sum a_i(-1)^{\sigma{(F_i,L_i)}}\theta(w_i)e_{L_i}.
\]
\end{Lemma}
\begin{proof}  Let $m_1,\ldots,m_t$ be a minimal homogeneous
system of generators of $M$ and let $n_1,\ldots,n_t$ be the
corresponding minimal homogeneous system of generators of $N$ with
$\theta(m_i)=n_i$. Let  $w_i=\sum b_{i j} m_{i j}x_{H_{i j}}$
where $b_{i j}\in K$ and where $F_i$ is a disjoint union of $G_{i
j}=\supp(m_{i j})$ and $H_{i j}$ for all $i j$. Then
\[
\theta(m)=\theta(\sum_ia_i(\sum_{j}b_{i j}m_{i j}x_{H_{i j
}})x_{L_i}= \theta(\sum_i\sum_{j }a_ib_{i j}m_{i j}x_{H_{i j}\cup
L_i})=\sum_i\sum_{j }(-1)^{\sigma(G_{i j},H_{i j}\cup L_i)}n_{i
j}e_{H_{i j}\cup L_i}.
\]
On the other hand
\begin{eqnarray*}
\sum a_i(-1)^{\sigma{(F_i,L_i)}}\theta(w_i)e_{L_i}&=&\sum_i\sum_{j
}(-1)^{\sigma(G_{i j}\cup H_{i j},L_i)} (-1)^{\sigma(G_{i j},H_{i
j})} a_ib_{i j}n_{i j}e_{H_{i j}}e_{L_i}\\&=&\sum_i\sum_{j
}(-1)^{\sigma(G_{i j},L_i)}(-1)^{\sigma( H_{i j},L_i
)}(-1)^{\sigma(G_{i j},H_{i j})}(-1)^{\sigma(H_{i j},L_i)}
a_ib_jn_{i j}e_{H_{i j}\cup L_i}\\&=&\sum_i\sum_{i
j}(-1)^{\sigma(G_{i j},H_{i j}\cup L_i)}n_{i j}e_{H_{i j}\cup
L_i}=\theta(m).
\end{eqnarray*}
\end{proof}

Let $W$ be an $\ZZ^n$-graded $K$-vector space. Then
$W^\ast=\Hom_K(W,K(-\mathbf{1}))$ is again a $\ZZ^n$-graded
$K$-vector space with the graded components
\[
(W^\ast)_\mathbf{a}=\Hom_K(W_{\mathbf 1-\mathbf a},K)\; \text {for
all}\; \mathbf a\in \ZZ^n.
\]
Here $\mathbf 1=(1,\ldots,1)$. Note that if $W$ is an
$\ZZ^n$-graded $E$-module, then  $W^\ast$ is also a $\ZZ^n$-graded
$E$-module. Furthermore if $W$ is a squarefree $E$-module, then
$W^\ast$ is again a squarefree $E$-module.

 In the category of
squarefree $E$-modules the graded $E$-dual is defined to be
$N^\vee=\Hom_E(N,E)$. Observe that $()^\vee$ is an exact
contravariant functor, see  \cite[5.1(a)]{AH}. Let $\phi\in
N^\vee$ and $n\in N$. Then $\phi(n)=\sum_{F\subseteq
[n]}\phi_F(n)e_F$ with $\phi_F(n)\in K$ for all $F\subseteq [n]$.
Therefore for each $F\subseteq[n]$ we obtain a $K$-linear map
$\phi_F: N\to K$.

The following theorem is important for the main result of this
paper.

\begin{Theorem} \cite{HHb}
\label{importent} Let $N$ be a $\ZZ^n$-graded $E$-module. The map
$\eta\:N^\vee \to N^\ast$, $\phi \to \phi_{[n]}$ is a functorial
isomorphism of $\ZZ^n$-graded $E$-modules. In particular if $N$ is
squarefree $E$-module, then $N^\vee$ is again
squarefree and $\eta$ is a functorial isomorphism of squarefree
$E$-modules.
\end{Theorem}

In \cite{R}, the Alexander dual of a squarefree
$S$-module is defined as follows:
\begin{Definition} Let $M\in SQ(S)$. Then $M^\vee=\mathbf G(\mathbf F(M)^\vee)$ is called
the Alexander dual of $M$.
\end{Definition}
Note that $$ SQ(S)\to SQ(S),\quad M\to M^\vee$$ is a
contravariant exact functor.

For example if $I\subset J$ are squarefree monomial ideals in $S$.
Let $\Delta$ and $\Gamma$ be simplicial complexes with
$I=I_\Delta$ and $J=I_{\Gamma}$.  Then $J/I$ is a squarefree
$S$-module and $(J/I)^\vee=I_{\Delta^\vee}/I_{\Gamma^\vee}$. In
particular if $\Delta$ is a simplicial complex on the vertex set
$[n]$ and $I_\Delta$ its Stanley-Reisner ideal, then
$(S/I_\Delta)^\vee=I_{\Delta^\vee}$ and
$(I_\Delta)^\vee=S/I_{\Delta^\vee}$.

\section{Prime filtrations and Alexander duality}
Let $S=K[x_1,\ldots,x_n]$ be the polynomial ring in $n$ variables
over a field $K$ and $M$ a finitely generated $\ZZ^n$-graded
$S$-module. It is known that the associated prime ideals of  $M$
are monomial ideals, and any monomial prime ideal is of the form
$P_F=(x_i\:i\in F)$ for some $F\subset [n]$. A chain $
0=M_0\subset M_1\subset\ldots\subset M_r=M$ of $\ZZ^n$-graded
submodules of $M$ such that ${M_i}/{M_{i-1}}\iso
{S}/{P_{F_i}}(-G_i)$ is called a prime filtration of $M$. If $M$
is a finitely generated $\ZZ^n$-graded $S$-module, then  a prime
filtration  of $M$  always exists, see \cite[Theorem 6.4]{Ma}.

We shall need the following

\begin{Lemma}
\label{notclear}
Let $M\subset M'$ be two squarefree
 $S$-modules and $N\subset N'$ be two squarefree $E$-modules.
\begin{enumerate}
\item[(a)]  If $M'/M \iso {S}/{P_F}(-G)$, then
$G\cap F=\emptyset$;
\item[(b)] We have
${M'}/{M} \iso {S}/{P_F}(-G)$ if and only if  ${\mathbf
F}(M')/{\mathbf F}(M)\iso E/P_{F\cup G}(-G)$, \\ where $P_{F\cup
G}=( e_j\: j\in F\cup G)$;
\item[(c)] We have $N'/N\iso E/P_{F\union G}(-G)$ if and only
 if ${\mathbf G}(N')/{\mathbf G}(N)\iso S/P_{F}(-G)$.
\end{enumerate}
\end{Lemma}

\begin{proof} (a) Suppose $G\cap F\neq \emptyset$.
Let $i\in G\cap F$ and let $f$ the homogeneous generator of $M'/M$.
Since $M'/M$ is squarefree, and since $\deg f=G$
it it follows that $x_if\neq 0$, a contradiction.

(b) Since ${\mathbf F}$ is an exact functor it suffices to show
that ${\mathbf F}(S/P_F(-G))=E/P_{F\union G}(-G)$. But this
follows immediately from the  Aramova-Avramov-Herzog complex
\cite[Theorem 1.3]{AAH} by which R\"{o}mer defined the functor
${\mathbf F}$ in \cite{R}.

(c) follows form (b) by using the fact that the functors
$\mathbf F$ and $\mathbf G$ are inverse to each other.
\end{proof}

Applying this lemma we get the following short exact sequence
$$0\to\mathbf F(M)\to \mathbf F (M')\to E/P_{F\cup G}(-G)\to 0.$$
Since $\Hom_E(-,E)$ is an contravariant exact functor, from the
above short exact sequence we obtain the short exact sequence
$$0\to\Hom_E(E/P_{F\cup G}(-G),E)\to {\mathbf F}(M')^\vee\to {\mathbf F}(M)^\vee\to 0.$$
 On the other hand $\Hom_E(E/P_{F\cup G}(-G),E)=\Hom_E(E/P_{F\cup G},E)(G)$. Since
 $$\Hom_E(E/P_{F\cup G},E)=0:_E P_{F\cup G}=(e_{F\cup G})\iso E/P_{F\cup G}(-F-G),$$
 one has
$\Hom_E(E/P_{F\cup G}(-G),E)\iso E/P_{F\cup G}(-F)$.

We conclude that the natural map
$$\alpha\: \mathbf F(M')^\vee\to
\mathbf F(M)^\vee$$ is an epimorphism  with $\Ker(\alpha)\iso E/P_{F\cup G}(-F)$.

\begin{Proposition}
\label{pedual} Let $N$ be a squarefree $E$-module and $N^\vee$ its
$E$-dual.
 Then there exists a chain $0\subset N_1\subset\ldots\subset N_t=N$ of
 squarefree submodules of $N$ with $N_i/N_{i-1}\iso
E/P_{F_i\cup G_i}(-G_i)$  if and only if there exists a chain
$0\subset H_1\subset\ldots\subset H_t=N^\vee$ of squarefree
submodule of $N^\vee$ with $H_i/H_{i-1}\iso E/P_{F_i\cup
G_i}(-F_i)$.
\end{Proposition}

\begin{proof} It is enough to prove one direction of the assertion, because $(N^\vee)^\vee=N$.
 Let $0=N_0\subset N_1\subset\ldots\subset N_t=N$
be a chain of squarefree $E$-modules with $N_i/N_{i-1}\iso
E/P_{F_i\cup G_i}(-G_i)$. From the observation above we see that for
 each $i$
there is an epimorphism $\alpha_i\: N_{t-i+1}^\vee\to
N_{t-i}^\vee$ with $\Ker\alpha_i\iso E/P_{F_i\cup G_i}(-F_i)$ .

Let $\beta_i\: N^\vee\to N_{t-i}^\vee$ be the epimorphism which is
defined by $\beta_i=\alpha_i\circ\alpha_{i-1}\circ\cdots
\circ\alpha_1$. Then
$$0\subset \Ker \beta_1\subset\cdots\subset \Ker \beta_t=N^\vee$$
is a filtration of $N^\vee$ by squarefree $E$-modules.
We only need to show that $\Ker\beta_i/\Ker\beta_{i-1}\iso \Ker
\alpha_i$. This follows from the Snake Lemma  applied to the following commutative diagram
\[
\begin{CD}
0@>>> &\Ker\beta_{i-1}& @ >\iota_1 >> & N^\vee &
@ >\beta_{i-1} >> & N^\vee_{t-i+1}& @>>> 0 \\
& & &@V \iota_2 VV & & @ V \id VV & & @ V \alpha_i VV \\
0@>>> &\Ker\beta_{i}& @ >\iota_3 >> & N^\vee &
 @ >\beta_i >>& N^\vee_i& @>>> 0
\end{CD}
\]
with exact rows,  where the $\iota_j$  are  inclusion maps.
\end{proof}

Now we can prove the corresponding  result for  squarefree $S$-modules.
 \begin{Theorem}
 \label{psdual} Let $M$ be a squarefree $S$-module and $M^\vee$
 its Alexander dual. Then there exists a chain
 $ 0\subset M_1\subset\cdots\subset M_r=M$ of
 squarefree submodules of $M$ with $M_i/M_{i-1}\iso
S/P_{F_i}(-G_i)$ if and only if  there exists a chain $0\subset
L_1\subset \cdots\subset L_r=M^\vee$ of squarefree submodules of
$M^\vee$ with $L_i/L_{i-1}\iso S/P_{G_i}(-F_i)$.
\end{Theorem}

\begin{proof} Again it is enough to prove one direction of the assertion, because $(M^\vee)^\vee=M$.
From the given chain of submodules of $M$ we get a
chain
$$0\subset \mathbf F(M_1)\subset\cdots\subset \mathbf F(M_r)=\mathbf F(M)$$
of  squarefree $E$-modules with $\mathbf F(M_i)/\mathbf
F(M_{i-1})\iso E/P_{F_i\cup G_i}(-G_i)$, see Lemma~\ref{notclear}(b). Therefore by Proposition
\ref{pedual} there exists a chain $0\subset N_1\subset \cdots
\subset N_{r-1}\subset N_r=(\mathbf F(M))^\vee$ of squarefree
$E$-modules with $N_i/N_{i-1}\iso E/P_{F_i\cup G_i}(-F_i)$.
This chain of squarefree $E$-modules induces the chain
$$0\subset \mathbf G(N_1)\subset\cdots\subset \mathbf
G(N_{r-1})\subset \mathbf G(N_r) =\mathbf G(\mathbf
F(M))^\vee)=M^\vee$$ of squarefree $S$-modules with $\mathbf
G(N_i)/\mathbf G(N_{i-1})\iso S/P_{G_i}(-F_i)$, see Lemma~\ref{notclear}(c).
\end{proof}

We now explain what Theorem~\ref{psdual} means in the special case
that $M=J/I$ where $I\subset J\subset S$ are squarefree monomial
ideals. To this end we introduce the following notation: let
$I\subset S$ be a squarefree monomial ideal and $\Delta$ be the
simplicial complex such that $I=I_\Delta$. We set
$\tilde{I}=I_{\Delta^\vee}$. Then $\tilde{\tilde{I}}=I$ since
$(\Delta^\vee)^\vee=\Delta$, and  if $I\subset J$ are two
squarefree monomial ideals, then $\tilde{J}\subset\tilde{I}$ and
$(J/I)^\vee=\tilde{I}/\tilde{J}$.

\begin{Corollary} \label{J/I}
Let $I\subset J$ be a squarefree monomial ideals. The following conditions  are
equivalent:
\begin{enumerate}
\item[(a)] $I=I_0\subset I_1\subset \ldots\subset I_{r-1}\subset
I_r=J$ is an $\NN^n$-graded  prime filtration of $J/I$ with
$I_i/I_{i-1}\iso S/P_{F_i}(-G_i)$.
\item[(b)]  $\tilde{J}=\tilde{I}_r\subset
\tilde{I}_{r-1}\subset\ldots\subset \tilde
I_1\subset\tilde{I}_0=\tilde I$ is an $\NN^n$-graded prime
filtration of $\tilde{I}/\tilde{J}=(J/I)^\vee$ with
$\tilde{I}_{i-1}/\tilde{I}_i\iso S/P_{G_i}(-F_i)$.
 \end{enumerate}
\end{Corollary}

\begin{proof} It is enough to prove the implication (a)\implies (b),
because $\tilde{\tilde{L}}=L$ for any squarefree monomial ideal
$L$. For the proof  we may assume that $r=1$, in other words
$J/I\iso S/P_F(-G)$. In this situation
$\tilde I/\tilde J=(J/I)^\vee\iso S/P_G(-F)$,  by Theorem \ref{psdual}.
\end{proof}

\section{Stanley decompositions and Alexander duality}
In \cite[Conjecture 5.1]{St} Stanley conjectured the following:
let $R$ be a finitely generated $\NN^n$-graded $K$-algebra (where
$R_0=K$ as usual), and let $M$ be a finitely generated
$\ZZ^n$-graded $R$-module. Then there exist finitely many
subalgebras $S_1,\ldots,S_t$ of $R$, each generated by
algebraically independent $\NN^n$-homogeneous elements of $R$, and
there exist $\ZZ^n$-homogeneous elements $m_1,\ldots,m_t$ of $M$,
such that $$M=\Dirsum_{i=1}^t m_iS_i$$ where $\dim S_i\geq\depth
M$ for all $i$, and where $m_iS_i$ is a free $S_i$-module (of rank
one). Moreover, if $K$ is infinite and under a given
specialization to an $\NN$-grading $R$ is generated by $R_1$, then
we can choose the ($\NN^n$-homogeneous) generators of each $S_i$
to lie in $R_1$.

Stanley's conjecture has been studied in several articles, see for
examples  \cite{Ap1}, \cite{Ap2}, \cite{So}, \cite{HSY},
\cite{AhP}, \cite{AnP}, \cite{NR}  and \cite{SX}.

We consider this conjecture in the case that $M$ is a finitely
generated $\ZZ^n$-graded $S$-module, where  $S=K[x_1,\ldots,x_n]$
is the polynomial ring in $n$ variables. Let $m\in M$ be a
homogeneous element and $Z\subset\{x_1,\ldots,x_n\}=X$. We denote
by $mK[Z]$ the $K$-subspace of $M$ generated by all homogeneous
elements of the form $mu$, where $u$ is a monomial in $K[Z]$. The
$K$-subspace $mK[Z]$ is called a {\em Stanley space of dimension
$|Z|$} if $mu\neq 0$ for any nonzero monomial $u\in K[Z]$.
According to \cite{HSY} the Stanley space $mK[Z]$ is called {\em
squarefree} if $m$ is a squarefree homogeneous element and
$\supp(m)\subset\supp(Z)=\{i\:x_i\in Z\}$ .

A decomposition $\mathcal D$ of $M$ as a finite direct sum of
Stanley spaces is called a {\em Stanley decomposition} of $M$. The
Stanley decomposition $\mathcal D$ of $M$  is called a {\em
squarefree Stanley decomposition} if all Stanley spaces in
$\mathcal D$ are squarefree Stanley spaces. The minimal dimension
of a Stanley space in the decomposition $\mathcal D$ is called the
{\em Stanley depth} of $\mathcal D$, denoted $\sdepth ({\mathcal
D})$. We set
$$\sdepth(M)=\max\{\sdepth({\mathcal D})\: {\mathcal D}\; \text{is a
Stanley decomposition of $M$}\}, $$ and call this number the {\em
Stanley depth} of $M$. For a squarefree module $M$ we denote by
$$\sqdepth(M)=\max\{\sdepth({\mathcal D})\: {\mathcal D} \text{ is
a squarefree Stanley decomposition of $M$}\}$$ the {\em squarefree
Stanley depth} of $M$. It is clear that
$\sqdepth(M)\leq\sdepth(M)$. With the above notation  Stanley's
conjecture says that  $ \depth(M)\leq\sdepth(M)$.

It is known  that the number of Stanley space of maximal dimension
is independent of the special Stanley decomposition, see
\cite[1018]{So}. Apel \cite{Ap2} showed  that if $I\subset S$ is a
monomial ideal, then $$\sdepth(S/I)\leq\min\{\dim(S/P)\:
P\in\Ass(S/I)\}.$$ The same result is true for any finitely
generated $\ZZ^n$-graded $S$-module $M$. Indeed, let $\mathcal
D=\Dirsum_{i=1}^tm_iK[Z_i]$ be a Stanley decomposition of $M$ such
that $\sdepth(\mathcal D)=\sdepth(M)$ and $P\in\Ass(M)$ an
associated prime such that $\dim(S/P)=\min\{\dim(S/Q)\:
Q\in\Ass(M)\}$. Since $P\in\Ass(M)$, there exists a nonzero
homogeneous element $m\in M$ such that $P=\Ann(m)$. On the other
hand since $0\neq m\in M$, there exists a unique $1\leq k\leq t$
such that $m\in m_kK[Z_k]$. It is enough to show that $Z_k\sect
P=\emptyset$. Let $m=m_kx^F$ for some $x^F\in K[Z_k]$. Suppose
that $Z_k\sect P\neq\emptyset$, and choose $x_i\in Z_k\sect P$.
Then $m_k(x^Fx_i)=mx_i=0$, a contradiction. This implies that
$|Z_k|\leq \dim(S/P)$. In particular,
$$\sdepth(M)=\sdepth(\mathcal D)\leq\dim(S/P)=\min\{\dim(S/Q)\:
Q\in\Ass(M)\}.$$

Let $I\subset S$ be a monomial ideal and $ I=I_0\subset I_1\subset
\cdots\subset I_r=S$ an $\NN^n$-graded prime filtration of $S/I$
with $I_i/I_{i-1}\iso S/P_{F_i}(-{\mathbf a}_i)$. It was shown in
\cite[page 398]{HP}  that this prime filtration of $S/I$ give us
the Stanley decomposition $S/I=\Dirsum_{i=1}^ru_ik[Z_i]$ of $S/I$,
where $Z_i=\{x_j\: j\not\in F_i\}$, and where
$u_i=x^{\mathbf{a}_i}$. This Stanley decomposition is called the
Stanley decomposition of $S/I$ corresponding to the given prime
filtration. With similar arguments one shows:

\begin{Proposition}
\label{nStanley} Let $M$ be a finitely generated $\ZZ^n$-graded
$S$-module. If $(0)=M_0\subset M_1\subset\cdots\subset M_r=M$ is a
is a prime filtration of $M$ such that $M_i/M_{i-1}\iso
S/P_{F_i}(-\mathbf{a}_i)$, then
$M\iso\Dirsum_{i=1}^rm_iK[Z_{F_i}]$ is a Stanley decomposition of
$M$ where $m_i\in M_{i}$ is a homogeneous element of degree
$\mathbf{a}_i$ such that $(M_{i-1}:_S m_i)=P_{F_i}$ and
$Z_{F_i}=\{x_j\:j\not\in F_i\}$.
\end{Proposition}

The following result is a generalization of \cite[Lemma 3.1]{HSY}.
Again we omit the proof because the arguments are analogue to
those  in the proof of \cite[Lemma 3.1]{HSY}.

\begin{Proposition}
\label{squarefree} Let $M$ be a finitely generated $\NN^n$-graded
$S$-module. Then $M$ has a squarefree Stanley decomposition if and
only if $M$ is a squarefree $S$- module.
\end{Proposition}

\begin{Remark} {\em In \cite[Proposition 2.5]{Y} Yanagawa proved that an
$\NN^n$-graded $S$-module $M$ is squarefree if and only if there
is a filtration of $\NN^n$-graded submodules $0\subset
M_1\subset\ldots\subset M_r=M$ of $M$ such that each quotient
$M_i/M_{i-1}\iso S/P_{F_i^c}(-F_i)$ for some $F_i\subset[n]$ where
$F_i^c=[n]\setminus F_i$. Yanagawa's result and Proposition
\ref{nStanley} implies one direction of Proposition
\ref{squarefree}.}
\end{Remark}

As a generalization of \cite[Theorem 3.3]{HSY} we have the
following. Again the same arguments like in the proof of
\cite[Theorem 3.3]{HSY} work also here.
\begin{Theorem}\label{sqdepth}
Let $M$ be an $\NN^n$-graded squarefree $S$-module. Then
$\sqdepth(M)=\sdepth(M)$.
\end{Theorem}

Let $E=K\langle e_1,\ldots, e_n\rangle$ be the exterior algebra
over an $n$-dimensional $K$-vector space $V$ and  $N$ a finitely
generated $\NN^n$-graded $E$-module. Let $n\in N$ be a homogeneous
element and $A\subset\{e_1,\ldots,e_n\}$. We set
$\supp(n)=\supp(\deg(n))$ and $\supp(A)=\{j\: e_j\in A\}$. We
denote by $nK\langle A\rangle$ the the $K$-subspace of $N$
generated by all homogeneous elements of the form $ne_F$, where
$e_F\in K\langle A\rangle$. If the elements $ne_F$ with $F\in
\supp(A)$ form a $K$-basis of $nK\langle A\rangle$, then we call
$nK\langle A\rangle $ a {\em Stanley space of dimension $|A|$}.

In case $N$ is a squarefree and  $nK\langle A\rangle\subset N$ is
a Stanley space we have that  $\supp(n)$ is squarefree and
$\supp(n)\sect \supp(A)=\emptyset$. A direct sum
$N=\Dirsum_{i=1}^tn_iK\langle A_i\rangle$  with Stanley spaces
$n_iK\langle A_i\rangle$ is called a {\em Stanley decomposition}
of $N$.

\begin{Proposition}
\label{Edual} Let $N$ be a squarefree  $E$-module, and $N^\vee$
the $E$-dual of $N$. Then there exists a Stanley decomposition
$N=\Dirsum_{i=1}^tn_iK\langle A_i\rangle$ of $N$ if and only if
there exists a  Stanley decomposition
$N^\vee=\Dirsum_{i=1}^tb_iK\langle A_i\rangle$ of $N^\vee$ with
$$\supp(b_i)=[n]\setminus(\supp(A_i)\cup\supp(n_i)).$$
\end{Proposition}

\begin{proof}
By Theorem \ref{importent} we have $N^\vee\iso
N^\ast=\Hom_K(N,K(-\mathbf 1))$. Hence we will show the assertion
for $N^\ast$. Since $N=\Dirsum_{i=1}^tn_iK\langle A_i\rangle$, as
an $\NN^n$-graded $K$-vector space one has
$N^\ast=\Dirsum_{i=1}^t(n_iK\langle A_i\rangle)^\ast$. Set
$\supp(n_i)= F_i$ and $\supp(A_i)=G_i$. Then $F_i\cap
G_i=\emptyset$ and the elements $n_ie_{H}$ with
$H\subseteq G_i$ form a $K$-basis of $n_iK\langle A_i\rangle$.
Consequently, the dual  elements $(n_ie_{H})^\ast$ form a $K$-basis of  $(n_iK\langle
A_i\rangle)^\ast$.

Let $b_i=(n_ie_{G_i})^\ast$ and $H,L\subseteq G_i$. Then
\[
(b_ie_H)(n_ie_L)=\pm b_i(n_ie_Le_H)=
\begin{cases}
0,\quad\text{if}\quad L\neq G_i\setminus H,\\
\pm 1,\quad\text{if}\quad L=G_i\setminus H,
\end{cases}
\]
and for any  $j\neq i$ and all $T\subset G_j$ one has
$(b_ie_H)(n_je_T)=\pm b_i(n_je_Te_H)=0$. This shows that
$b_ie_H=\pm(n_ie_{G_i\setminus H})^\ast$ for any $H\subset G_i$.
Therefore $(n_iK\langle A_i\rangle)^\ast=b_iK\langle A_i\rangle$
and $N^\ast=\Dirsum_{i=1}^tb_iK\langle A_i\rangle$.
\end{proof}

Let $M$ be a squarefree $S$-module and let $N$ be its
corresponding squarefree $E$-module. In Section~1 we showed that
there is an isomorphism  $\theta\: M_{\sq}\to N$ of graded
$K$-vector spaces. We will use this isomorphism to  describe
in the  next lemma the relationship between squarefree Stanley
decompositions of $M$ and Stanley decompositions of $N$.

\begin{Lemma}
\label{relation} {\em (a)} Let $M=\Dirsum_{i=1}^tm_iK[Z_i]$ be a
squarefree Stanley decomposition of $M$ and $$A_i=\{e_j\: j\in
\supp(Z_i)\setminus \supp(m_i)\}.$$ Then
$N=\Dirsum_{i=1}^tn_iK\langle A_i\rangle$ is a Stanley
decomposition of $N$, where $n_i=\theta(m_i)\in N$ for
$i=1,\ldots,t$.\\
{\em (b)} Conversely, if $N=\Dirsum_{i=1}^tn_iK\langle A_i\rangle$
is a Stanley decomposition of $N$ and $$Z_i=\{x_j\: j\in
\supp(A_i)\union \supp(n_i)\}.$$ Then $M=\Dirsum_{i=1}^tm_iK[Z_i]$
is a squarefree Stanley decomposition of $M$, where
$m_i=\theta^{-1}(n_i)\in M$ for $i=1,\ldots,t$.
\end{Lemma}

\begin{proof} (a): Since $M=\Dirsum_{i=1}^tm_iK[Z_i]$, one has
$$\Union_{i=1}^t\{m_ix_{F}\: F\subset
\supp(A_i)\}$$ forms a $K$-basis of $M_{\sq}$, and hence
$$\Union_{i=1}^t\{\theta(m_ix_{F})\: F\subset \supp(A_i)\}$$ forms
a $K$-basis of $N$. By Lemma \ref{alsoterrible} we have
$\theta(m_ix_{F})=(-1)^{\sigma(G_i,F)}n_ie_F$, where
$G_i=\supp(m_i)$.  Therefore
\[
\Union_{i=1}^t\{n_ie_F\: F\subset A_i\}
\]
forms a $K$-basis of $N$.

(b): Let $x^{\mathbf a}\in K[Z_i]$. We can write $x^{\mathbf
a}=x^\mathbf {a'}x^{\mathbf b}$ where ${\mathbf b}\in\NN^n$ is a
squarefree vector with $F=\supp(\mathbf {b})\subset \supp(A_i)$.
Then
\[
m_ix^\mathbf{a}=(m_ix^\mathbf{b})x^\mathbf{a'}=(-1)^{\sigma(G_i,F)}\theta^{-1}(n_ie_F)x^\mathbf{a'}.
\]
Since $\theta^{-1}(n_ie_F)\neq 0$ and since $M$ is squarefree and
$\supp(\mathbf{a'})\subset\supp(\theta^{-1}(n_ie_F))$, one has
$m_ix^\mathbf{a}\neq 0$. Therefore
\[
\Union_{i=1}^t\{m_ix^\mathbf{a}\: x^\mathbf{a}\in K[Z_i]\}
\]
forms a $K$-basis of $M$.
\end{proof}

 Now we will present the main result of this section.

\begin{Theorem}
 \label{dual}
Let $M$ be a squarefree $S$-module, and $M^\vee$ its Alexander
dual. Then there exists a squarefree Stanley decomposition
$M=\Dirsum_{i=1}^tm_iK[Z_i]$  of $M$ if and only if there exists a
squarefree Stanley decomposition $M^\vee=\Dirsum_{i=1}^tv_iK[W_i]$
 of $M^\vee$ with $\supp(v_i)=[n]\setminus\supp(Z_i)$
and $W_i=~\{x_j\:j\in[n]\setminus\supp(m_i)\}$.
\end{Theorem}

\begin{proof} Let $M=\Dirsum_{i=1}^tm_iK[Z_i]$ be a squarefree Stanley decomposition
of $M$. If we set $F_i=\supp(m_i)$ and $G_i=\supp(Z_i)\setminus
F_i $, then $F_i\cap G_i=\emptyset$. Let $N$ be the squarefree
$E$-module corresponding to $M$. Then by Lemma \ref{relation}(a),
$N$ has a Stanley decomposition
$$N=\Dirsum_{i=1}^t n_iK\langle A_i\rangle$$ where
$n_i=\theta(m_i)$  and $G_i=\supp(A_i)$. Hence by Proposition
\ref{Edual}, $N^\vee$ has a decomposition
$N^\vee=\Dirsum_{i=1}^tb_iK\langle A_i\rangle$ with
$\supp(b_i)=[n]\setminus (G_i\cup F_i)$. Therefore by Lemma
\ref{relation}(b),  $M^\vee$ the corresponding squarefree
$S$-module to $N^\vee$ has a decomposition as required.
\end{proof}

Associated to any finitely generated $\NN^n$-graded $S$-module $M$
is a {\em minimal free $\ZZ^n$-graded resolution}
\[
0\to\Dirsum_jS(-\mathbf{a}_j)^{\beta_{r,j}(M)}\to\cdots\to
\Dirsum_jS(-\mathbf{a}_j)^{\beta_{1,j}(M)}\to\Dirsum_jS(-\mathbf{a}_j)^{\beta_{0,j}(M)}\to
 0
\]
 where $S(-\mathbf{a}_j)$ denote the $\ZZ^n$-graded $S$-module obtained by
 shifting the degrees of $S$ by $\mathbf{a}_j$. The number $\beta_{i,j}(M)$
 is the {\em $ij$-th graded Betti number} of $M$. The regularity of
 $M$ is
 \[
 \reg(M)=\max\{|\mathbf{a}_j|-i\:\text {for all}\; i,j\}.
 \]

 Let $M$ be a
squarefree $\NN^n$-graded $S$-module. If Stanley's conjecture
holds for $M$, then by Theorem \ref{sqdepth} we may assume that
there exists a squarefree Stanley decomposition
$M=\Dirsum_{i=1}^tm_iK[Z_i]$ of $M$ such that $|Z_i|\geq
\depth(M)$. Also by Theorem \ref{dual} there exists a squarefree
Stanley decomposition $M^\vee=\Dirsum_{i=1}^tv_iK[W_i]$ of the
Alexander dual of $M$ such that $|\deg(v_i)|=n-|Z_i|\leq
n-\depth(M)$. On the other hand $\projdim(M)=\reg(M^\vee)$, see
\cite[Corollary 3.7]{T}.
 Since $\depth(M)+\projdim(M)=~n$, see
 \cite[Theorem 1.3.3]{BH}, we have $|\deg(v_i)|\leq\reg(M^\vee)$ for all $i$. Therefore
we will get the following:

\begin{Corollary}
\label{1} Let $M$ be a squarefree $\NN^n$-graded $S$-module and
$M^\vee$ its Alexander dual. Then Stanley's conjecture holds for
$M$ if and only if $M^\vee$ has a squarefree Stanley decomposition
$M^\vee=\Dirsum_{i=1}^tv_iK[W_i]$
 with
$|\deg(v_i)|\leq\reg(M^\vee)$ for all $i$.
\end{Corollary}

In the case that $I\subset S$ is a monomial ideal and  $M=S/I$ or
$M=I$, then we may consider the standard grading for $S$ and $M$
by setting $\deg(x_i)=1$ for $i=1,\ldots,n$. In this case a
minimal graded free resolution of $I$ is
\[
 0\to\Dirsum_jS(-j)^{\beta_{r,j}(M)}\to\cdots\to\Dirsum_jS(-j)^{\beta_{1,j}(M)}
 \to\Dirsum_jS(-j)^{\beta_{0,j}(M)}\to
 I\to 0.
 \]
Suppose that all  monomial minimal generators of $I$ are of degree
$d$. Then $I$ has a {\em linear resolution} if for all $i\geq 0$,
$\beta_{i,j}=0$ for all $j\neq i+d$.

Let $F\subset G\subset [n]$. We denote the  interval $\{H \:
F\subseteq H\subseteq G\}$  by $[F,G]$. A partition $\mathbf
P:\;\Delta=\bigcup_{i=1}^t[F_i,G_i]$ of $\Delta$ is a disjoint
union of intervals of $\Delta$. A simplicial complex $\Delta$ is
called partitionable if there is a partition  $\mathbf
P:\;\Delta=\bigcup_{i=1}^t[F_i,G_i]$ of $\Delta$ such that
$\{G_1,\ldots,G_t\}$ is the set of  facets of $\Delta$. In \cite
{St1}  Stanley conjectured that any Cohen-Macaulay simplicial
complex is partitionable, see also \cite{St2}. In \cite[Corollary
3.5]{HSY} it was shown that this conjecture is a special case of
Stanley's conjecture on Stanley decompositions. Indeed, the
authors proved that if $\mathbf
P:\;\Delta=\bigcup_{i=1}^t[F_i,G_i]$ is a partition of $\Delta$,
then $\mathcal D(\mathbf P):\;S/I_\Delta=\Dirsum
x_{F_i}K[Z_{G_i}]$ is a squarefree Stanley decomposition of
$S/I_\Delta$, where $x_{F_i}=\prod_{j\in F_i}x_j$ and
$Z_{G_i}=\{x_j\: j\in G_i\}$. Hence we get the following
corollary.

\begin{Corollary} \label {pl}
A Cohen-Macaulay simplicial complex $\Delta$ is partitionable if
and only if $I_{\Delta^\vee}$ has a squarefree Stanley
decomposition $I_{\Delta^\vee}=\Dirsum_{i=1}^tu_iK[Z_i]$ such that
$\{u_i,\ldots,u_t\}=G(I_{\Delta^\vee})$.
\end{Corollary}
\begin{proof} By Eagon-Reiner \cite{ER} $\Delta$ is
Cohen-Macaulay if and only if  $I_{\Delta^\vee}$ has a linear
resolution. Also by a result of Terai \cite{T}
$\projdim(S/I_\Delta)=\reg(I_{\Delta^\vee})$ for any simplicial
complex $\Delta$.

On the other hand by  Corollary \ref{1} the Cohen-Macaulay
simplicial complex $\Delta$ is partitionable if and only if
$I_{\Delta^\vee}$ has a squarefree Stanley decomposition
$I_{\Delta^\vee}=\Dirsum_{i=1}^tu_iK[Z_i]$ such that $\deg u_i\leq
\reg(I_{\Delta^\vee})= d$, where $d$ is the degree of any minimal
monomial generator of $I_{\Delta^\vee}$. Since $u_i\in
I_{\Delta^\vee}$, one has $\deg(u_i)\geq d$ for all $i$. This
shows that $u_i\in G(I_{\Delta^\vee})$ and hence
$\{u_i,\ldots,u_t\}\subset G(I_{\Delta^\vee})$. The other
inclusion is obvious.
\end{proof}

 Corollary \ref{pl} shows
that Stanley's conjecture which says that any Cohen-Macaulay
simplicial complex is patitionable  is equivalent to say that any
squarefree monomial ideal  $I\subset S$  which has a linear
resolution has  a  Stanley decomposition
$I=\Dirsum_{i=1}^tu_iK[Z_i]$ such that $\{u_1,\ldots,u_t\}=G(I)$.

This results lead us to  make the following conjecture which in
the case of squarefree $\NN^n$-graded $S$-module is equivalent to
Stanley's conjecture on Stanley decompositions.

\begin{Conjecture}
\label{reg}
 Let $S=K[x_1,\ldots,x_n]$,
 and let $M$ be a finitely generated $\ZZ^n$-graded $S$-module.
Then there exists a Stanley decomposition
$$M=\Dirsum_{i=1}^t m_iK[Z_i],$$
of $M$, where $| m_i|\leq\reg M$ for all $i$.
\end{Conjecture}

Let $\mathcal D$  be a Stanley decomposition of $M$.  We call the maximal
  $|m_i|$ in  $\mathcal D$  the {\em $h$-regularity
} of $\mathcal D$, and denote it by  $\hreg ({\mathcal D})$. Maclagan
and Smith \cite[Remark 4.2]{MS} proved that $\hreg({\mathcal
D})\geq \reg(M)$ in the case that $M=S/I$ , where $I$ is a monomial
ideal, and $\mathcal D$ is a Stanley filtration. We set $
\hreg(M)=\min\{\hreg({\mathcal D})\: {\mathcal D}\; \text{is a
Stanley decomposition of $M$}\}, $
 and call this number
the {\em $h$-regularity } of $M$.
With the  notation introduced  our conjecture says that
$\hreg(M)\leq \reg(M)$.

Let $M$ be a finitely generated $\NN^n$-graded $S$-module which is
generated by homogeneous elements $n_1,\ldots,n_s$. It is clear
that $|n_i|\leq\reg(M)$ for $i=1,\ldots,s$. We want to show that
$|n_i|\leq\hreg(M)$ for $i=1,\ldots,s$. Let $\mathcal
D=\Dirsum_{i=1}^tm_iK[Z_i]$ be a Stanley decomposition of $M$ such
that $\hreg(\mathcal D)=\hreg(M)$, and
$|n_r|=\max\{|n_i|\:i=1,\ldots,s\}$. Since $n_r\in M$ is a
homogeneous element, there exists a $j\in[t]$ such that $n_r\in
m_jK[Z_j]$. On the other hand $m_j\in M$ and $n_r$ is a generator.
 Therefore we have $m_j=n_r$ and
 $|n_r|=|m_j|\leq \hreg(\mathcal D)$.

Let $I\subset S=K[x_1,\ldots,x_n]$ be a monomial ideal. Apel
\cite{Ap1} proved that if $\depth(I)\leq 2$ or $n\leq 3$, then
Stanley's conjecture holds for $I$. Also if $n\leq 3$, then
Stanley's conjecture holds for $S/I$, see \cite{Ap2} or \cite{So}.
Furthermore in \cite{HSY} the authors showed that Stanley'
conjecture holds for $S/I$ if $I$ is a complete intersection,
$S/I$ is Cohen--Macaulay of codimension 2, or $S/I$ is Gorenstein of
codimension 3. If $I=I_\Delta$ is a squarefree monomial ideal,
then $\projdim(I_\Delta)=\reg(S/I_{\Delta^\vee})$. The discussions above
together with  Corollary \ref{1} yield the following:

\begin{Corollary} Let $I\subset S=K[x_1,\ldots,x_n]$ be a squarefree monomial ideal.
Then
\begin{enumerate}
\item[(i)]  Conjecture \ref{reg} holds for $I$
and for $S/I$ if $n\leq 3$;
\item[(ii)]   Conjecture \ref{reg}
holds for $S/I$ if $\reg(S/I)\geq n-2$;
\item[(iii)] Conjecture \ref{reg} holds for $I$ if $\reg(I)=2$.
\end{enumerate}
\end{Corollary}

Let $I=(u_1,\ldots,u_m)$ be a monomial ideal in $S$. According to
\cite {HT}, the monomial ideal $I$ has linear quotients if one can
order the set of minimal generators of $I$,
$G(I)=\{u_1,\ldots,u_m\}$, such that the ideal
$(u_1,\ldots,u_{i-1}):u_i$ is generated by a subset of the
variables for $i=2,\ldots,m$.

Assume that $I=(u_1,\ldots,u_m)$ is a monomial ideal which has
linear quotients with respect to the given order. Set
$I_i=(u_1,\ldots,u_{i-1}):u_i$, $Z_i=X\setminus G(I_i)$ for
$i=2,\ldots,m$ and $Z_1=X$. We denote $r_i=|G(I_i)|$ for
 $i=2,\ldots,m$
  and $r(I)=\max\{r_i \:i=2,\ldots,s\}$.  By \cite[page 539]{HH} $\depth(I)=n-r(I)$.

\begin{Corollary}\label{linear1} Let $I\subset S$ be a monomial
ideal with linear quotients. Then  Stanley's conjecture on Stanley
decompositions holds for $I$.
\end{Corollary}

\begin{proof}  Suppose $I=(u_1,\ldots,u_m)$ has linear
quotients with respect to the given order. Then  $\mathcal G :\;
(0)\subset J_1=(u_1)\subset \ldots\subset
J_{m-1}=(u_1,\ldots,u_{m-1})\subset J_m=I$ is a prime filtration
of $I$. Hence by  Proposition \ref{nStanley} $\mathcal D
=\Dirsum_{i=1}^su_iK[Z_i]$ is a Stanley decomposition of $I$ with
 $\sdepth(\mathcal D)=n-r(I)=\depth(I)$.
\end{proof}

In the decomposition above of $I$, all $u_i$ are the minimal
monomial generators of $I$. Therefore we have
\begin{Corollary} \label{linear2}If $I\subset S$ is a monomial ideal which has
linear quotient, then Conjecture \ref{reg} holds for $I$.
\end{Corollary}

In \cite{HHZ} it was shown that if $I$ is monomial ideal with
2-linear resolution, then $I$ has linear quotients. Therefore
Stanley's conjecture on Stanley decompositions and  Conjecture
\ref{reg} holds for such monomial ideals.

\end{document}